\documentclass[fleqn, 11pt]{article}
\usepackage{amsmath, amssymb, amsthm, graphicx}

\theoremstyle{definition}
\newtheorem{theorem}{Theorem}[section]

\newtheorem{lemma}[theorem]{Lemma}
\newtheorem{corollary}[theorem]{Corollary}

\newtheorem{definition}[theorem]{Definition}
\newtheorem{Remark}[theorem]{Remark}
\newenvironment{remark}{\begin{Remark}\rm}{\end{Remark}}

\newtheorem{Example}[theorem]{Example}
\newtheorem{Question}[theorem]{Question}

\numberwithin{equation}{section}
\title{Polygonal rotopulsators of the curved $n$-body problem. }

\author{Pieter Tibboel \\
Department of Mathematical Sciences\\
Xi'an Jiaotong-Liverpool University\\
Suzhou, China\\
Pieter.Tibboel@xjtlu.edu.cn}

\begin{document}
\maketitle
\begin{abstract}
  We revisit polygonal positive elliptic rotopulsator solutions and polygonal negative elliptic rotopulsator solutions of the $n$-body problem in $\mathbb{H}^{3}$ and $\mathbb{S}^{3}$ and prove existence of these solutions, prove that the masses of these rotopulsators have to be equal if the rotopulsators are of nonconstant size and show that the number of negative elliptic relative equilibria of this type is finite, as is the number of positive elliptic relative equilibria if an upper bound on the size of the relative equilibrium is imposed. Additionally, we prove that a class of negative hyperbolic rotopulsators is in fact a subclass of the class of polygonal negative elliptic rotopulsators.
\end{abstract}
\maketitle

\section{Introduction}
  By $n$-body problems we mean problems where we are to determine the dynamics of a number of $n$ point masses as dictated by a system of ordinary differential equations. The $n$-body problem in spaces of constant Gaussian curvature, or curved $n$-body problem for short, generalises the classical, or Newtonian $n$-body problem to spaces of constant Gaussian curvature and is defined as follows:
  \begin{definition}
    Let $\sigma=\pm 1$. The $n$-body problem in spaces of constant Gaussian curvature is the problem of finding the dynamics of point masses \begin{align*}q_{1},...,\textrm{ }q_{n}\in\mathbb{M}_{\sigma}^{3}=\{(x_{1},x_{2},x_{3},x_{4})\in\mathbb{R}^{4}|x_{1}^{2}+x_{2}^{2}+x_{3}^{2}+\sigma x_{4}^{2}=\sigma\},\end{align*} with respective masses $m_{1}>0$,..., $m_{n}>0$, determined by the system of differential equations
  \begin{align}\label{EquationsOfMotion Curved}
   \ddot{q}_{i}=\sum\limits_{j=1,\textrm{ }j\neq i}^{n}\frac{m_{j}(q_{j}-\sigma(q_{i}\odot q_{j})q_{i})}{(\sigma -\sigma(q_{i}\odot q_{j})^{2})^{\frac{3}{2}}}-\sigma(\dot{q}_{i}\odot\dot{q}_{i})q_{i},\textrm{ }i\in\{1,...,\textrm{ }n\},
  \end{align}
  where for $x$, $y\in\mathbb{M}_{\sigma}^{3}$  the product $\cdot\odot\cdot$ is defined as
  \begin{align*}
    x\odot y=x_{1}y_{1}+x_{2}y_{2}+x_{3}y_{3}+\sigma x_{4}y_{4}.
  \end{align*}
  \end{definition}
  The curved $n$-body problem for $n=2$ goes back as far as the 1830s, but a working model for the $n\geq 2$ case was not found until 2008 by Diacu, P\'erez-Chavela and Santoprete (see \cite{DPS1}, \cite{DPS2} and \cite{DPS3}). This breakthrough then gave rise to further results for the $n\geq 2$ case in \cite{D1}--\cite{DK} and \cite{DPo}--\cite{ZZ}. See \cite{DK}, \cite{DPS1}, \cite{DPS2} and \cite{DPS3} for a historical overview.
  Solutions to an $n$-body problem where the point masses describe a configuration that maintains the same shape and size over time are called relative equilibria. A rotopulsator, or rotopulsating orbit is a solution of the curved $n$-body problem for which the shape of the configuration of the point masses stays the same over time, but the size may change. An important reason to study the curved $n$-body problem, and relative equilibria and rotopulsators in particular, is to identify orbits that are unique to a particular space (see \cite{DK}). For example: Diacu, P\'erez-Chavela and Santoprete (see \cite{DPS1}, \cite{DPS2}) showed that rotopulsators (called homographic orbits in those papers) that have an equilateral triangle configuration and unequal masses, only exist in spaces of zero curvature. As the Sun, Jupiter and the Trojan asteroids form the vertices of an equilateral triangle, the region between these three objects likely has zero curvature. Rotopulsators were first introduced in \cite{DK}, where it was proven that there are five different types of rotopulsators, two (positive elliptic and positive elliptic-elliptic rotopulsators) for the positive curvature case (spheres) and three (negative elliptic, negative elliptic-hyperbolic and negative hyperbolic rotopulsators) for the negative curvature case (hyperboloids) (see \cite{DK}). For these five different types it was proven in \cite{DT} for $n=4$ that if the rotopulsators have rectangular configurations, they have to be squares. In this paper we will further investigate a result on rotopulsators for general $n$ for four of these classes, but before getting into specifics, we will need a precise definition of these four classes, namely positive elliptic rotopulsators, negative elliptic rotopulsators, negative elliptic-hyperbolic rotopulsators and negative hyperbolic rotopulsators:
  \begin{definition}\label{Definition negative hyperbolic}
  Let $q_{1}$,...,$q_{n}$ be a solution of (\ref{EquationsOfMotion Curved}) for which the shape of the point configuration may rotate, or change size, but otherwise remains unchanged. We will write
  \begin{align*}
    q_{i}=\begin{pmatrix}
      q_{i1} \\
      q_{i2} \\
      q_{i3} \\
      q_{i4}
    \end{pmatrix},\textrm{ }i\in\{1,...,n\}
  \end{align*}
  where $q_{i1}$, $q_{i2}$, $q_{i3}$ and $q_{i4}$ are the components of the vector $q_{i}$. Let
  \begin{align*}
    T(x)=\begin{pmatrix}
      \cos{x} & -\sin{x} \\
      \sin{x} & \cos{x}
    \end{pmatrix}\textrm{ and } S(x)=\begin{pmatrix}
      \cosh{x} & \sinh{x} \\
      \sinh{x} & \cosh{x}
    \end{pmatrix}
  \end{align*}
  be $2\times 2$ matrices. If there exist positive, twice differentiable functions $r_{1}$,...,$r_{n}$, a twice differentiable function $\theta$ and constants $0\leq\alpha_{1}<...<\alpha_{n}<2\pi$ such that
  \begin{align*}
    \begin{pmatrix}
      q_{i1}\\
      q_{i2}
    \end{pmatrix}=r_{i}T(\theta+\alpha_{i})\begin{pmatrix}
      1 \\
      0
    \end{pmatrix},\textrm{ }i\in\{1,...,\},
  \end{align*}
  then we call $q_{1}$,...,$q_{n}$ a positive elliptic rotopulsator if $\sigma=1$ and a negative elliptic rotopulsator if $\sigma=-1$ (see \cite{DK}).
    If $\sigma=-1$ and there exist scalar, twice differentiable functions $\phi$, $\rho_{i}\geq 0$, $i\in\{1,...,n\}$ and constants  $\beta_{1},...,\beta_{n}\in\mathbb{R}$, then if
    \begin{align}\label{q_i negative hyperbolic}
      \begin{pmatrix}
        q_{i3} \\
        q_{i4}
      \end{pmatrix}=\rho_{i}S(\phi+\beta_{i})\begin{pmatrix}
        0 \\
        1 \\
      \end{pmatrix},\textrm{ }i\in\{1,...,n\},
    \end{align}
    we call $q_{1}$,...,$q_{n}$ a negative hyperbolic rotopulsator (see \cite{DK}).
    If $q_{1}$,...,$q_{n}$ is both a negative elliptic and negative hyperbolic rotopulsator, then $q_{1}$,...,$q_{n}$ is called a negative elliptic-hyperbolic rotopulsator. If the $r_{i}$ and $\rho_{i}$ are constant and $\phi$ and $\theta$ are linear functions, then we speak of positive elliptic relative equilibria, negative elliptic equilibria, negative hyperbolic relative equilibria and negative elliptic-hyperbolic relative equilibria.
  \end{definition}
  In \cite{T2} positive elliptic rotopulsators and negative elliptic rotopulsators were investigated under the restriction that $q_{i3}$ and $q_{i4}$, $i\in\{1,...,n\}$, are independent of $i$ for all $i\in\{1,...,n\}$ (and consequently that $r_{i}=r$ is independent of $i$ as well) and it was proven that if such rotopulsators are of nonconstant size, the shape of their configurations has to be a regular polygon. In this paper we will revisit these two classes, prove existence of those classes, prove that the masses of such rotopulsators are equal if they are of nonconstant size, prove that the same results are true for negative hyperbolic rotopulsators under the weaker condition that the $\rho_{i}$ are independent of $i$, show that these negative hyperbolic rotopulsators are in fact a subclass of the class of negative elliptic rotopulsators for which $q_{i3}$ and $q_{i4}$ are independent of $i$ and finally show that the number of negative elliptic relative equilibria and negative hyperbolic relative equilibria in these classes is finite and that the number of positive elliptic relative equilibria mentioned previously is finite under the condition that $r<\frac{2}{5}\sqrt{5}$. Specifically, we will prove the following results:
  \begin{theorem}\label{Main Theorem 1}
    Let $q_{1},...,q_{n}$ be a negative hyperbolic rotopulsator. If $\rho_{i}=\rho$ is the same function for all $i\in\{1,...,n\}$, then there exists a constant $\beta$ such that
    \begin{align}\label{proven}
    \begin{pmatrix}
        q_{i3} \\
        q_{i4}
      \end{pmatrix}=\rho S(\beta+\phi)\begin{pmatrix}
        0\\
        1
      \end{pmatrix}.
    \end{align}
  \end{theorem}
  \begin{theorem}\label{Main Theorem 2}
    Let $q_{1},...,q_{n}$ be either a positive elliptic, or a negative elliptic rotopulsator for which $q_{i3}=z_{1}$ and $q_{i4}=z_{2}$ and consequently $r_{i}=r$ are independent of $i$ for all $i\in\{1,...,n\}$. If $r$ is not constant, then the point masses form a regular polygon and all masses are equal.
  \end{theorem}
  \begin{corollary}\label{Main Theorem 3}
     Positive elliptic, negative elliptic and negative hyperbolic rotopulsators $q_{1},...,q_{n}$ for which $q_{i3}$ and $q_{i4}$ are independent of $i$, $i\in\{1,...,n\}$, exist. If they are not relative equilibria, then they have to have equal masses and configurations that are regular polygons. If they are relative equilibria, then for $\sigma=-1$, for each fixed set of masses, there exists at most one relative equilibrium. If $\sigma=1$ and $r<\frac{2}{5}\sqrt{5}$, then for each fixed set of masses there exists at most one relative equilibrium. Finally, negative hyperbolic rotopulsators for which $\rho_{i}$, $i\in\{1,...,n\}$, is independent of $i$, are negative elliptic rotopulsators for which $q_{i3}$ and $q_{i4}$ are independent of $i$.
  \end{corollary}
  \begin{remark}
    (\ref{proven}) was proven in \cite{ZZ} for the case that the negative hyperbolic rotopulsator $q_{1},...,q_{n}$ is a relative equilibrium. Theorem~\ref{Main Theorem 1} in combination with Corollary~\ref{Main Theorem 3} shows that this result also holds for all negative hyperbolic rotopulsators for which $\rho_{i}$, $i\in\{1,...,n\}$, is independent of $i$.
  \end{remark}
  \begin{remark}
    Existence of polygonal negative hyperbolic rotopulsators was essentially already proven in Theorem~1 of \cite{D2}, but for completeness we have added a proof in this paper as well.
  \end{remark}
  \begin{remark}
    In \cite{DT}, in Theorem~6 and Theorem~7, it was stated that for $n=4$ rectangular negative hyperbolic rotopulsators and rectangular negative elliptic hyperbolic rotopulsators do not exist. These statements are not in conflict with Theorem~\ref{Main Theorem 2}, as the third and fourth coordinates of the point masses of the negative hyperbolic rotopulsators in \cite{DT} are constructed to be coordinates of distinct points on a hyperbola, while we do not impose that restriction in this paper.
  \end{remark}
  \begin{remark}
    In \cite{PS} nonexistence of polygonal hyperbolic relative equilibria was proven for the case that all masses are equal and the space on which the problem is defined is $\mathbb{H}^{2}$. The dynamics for $\mathbb{H}^{3}$ are richer than for $\mathbb{H}^{2}$, which is why we do find existence of solutions in this paper. 
  \end{remark}

  The remainder of this paper is constructed as follows: We will first prove a lemma needed to prove our main results in section~\ref{Background}, after which we will prove Theorem~\ref{Main Theorem 1} in section~\ref{Section proof of main theorem 1}, Theorem~\ref{Main Theorem 2} in section~\ref{Section proof of main theorem 2} and Corollary~\ref{Main Theorem 3} in section~\ref{Section proof of main theorem 3}.
  \section{Background theory}\label{Background}
  To prove Theorem~\ref{Main Theorem 1}, Theorem~\ref{Main Theorem 2} and Corollary~\ref{Main Theorem 3}, we will need the following lemma, which was proven in a more general setting in \cite{DK}, but as the proof for our particular case is not particularly long, we give a proof here as well:
  \begin{lemma}\label{Lemmaaaaa}
    If $q_{1}$,...,$q_{n}$ is a negative hyperbolic rotopulsator, or a negative elliptic-hyperbolic rotopulsator as in Definition~\ref{Definition negative hyperbolic}, with functions $\rho_{i}=\rho$ and $\phi_{i}=\phi$, $i\in\{1,...,n\}$, independent of $i$  then \begin{align*}2\rho'\phi'+\rho\phi''=0.\end{align*}
  \end{lemma}
  \begin{proof}
   Using the wedge product, it was proven in \cite{D3} that
   \begin{align*}
     \sum\limits_{i=1}^{n}m_{i}q_{i}\wedge\ddot{q}_{i}=\mathbf{0},
   \end{align*}
   where $\mathbf{0}$ is the zero bivector.
   If $e_{1}$, $e_{2}$, $e_{3}$ and $e_{4}$ are the standard basis vectors in $\mathbb{R}^{4}$, then
   \begin{align}\label{Wedge q}
     0e_{3}\wedge e_{4}&=\sum\limits_{i=1}^{n}m_{i}(q_{i3}\ddot{q}_{i4}-q_{i4}\ddot{q}_{i3})e_{3}\wedge e_{4}.
   \end{align}
   As $q_{i3}=\rho\sinh{(\beta_{i}+\phi)}$ and $q_{i4}=\rho\cosh{(\beta_{i}+\phi)}$ by Definition~\ref{Definition negative hyperbolic}, we have that
   \begin{align}\label{Wedge q2}
     & q_{i3}\ddot{q}_{i4}-q_{i4}\ddot{q}_{i3}=-\det{\begin{pmatrix}
       q_{i4} & \ddot{q}_{i4} \\
       q_{i3} & \ddot{q}_{i3}
     \end{pmatrix}}=-\det{\begin{pmatrix}
       \rho\begin{pmatrix}
         \cosh{(\beta_{i}+\phi)} \\
         \sinh{(\beta_{i}+\phi)}
       \end{pmatrix} & \left(\rho\begin{pmatrix}
         \cosh{(\beta_{i}+\phi)} \\
         \sinh{(\beta_{i}+\phi)}
       \end{pmatrix}\right)''
     \end{pmatrix}}.
   \end{align}
   Because
   \begin{align*}
     \left(\rho\begin{pmatrix}
         \cosh{(\beta_{i}+\phi)} \\
         \sinh{(\beta_{i}+\phi)}
       \end{pmatrix}\right)''&=\rho''\begin{pmatrix}
         \cosh{(\beta_{i}+\phi)} \\
         \sinh{(\beta_{i}+\phi)}
       \end{pmatrix}+2\rho'\phi'\begin{pmatrix}
         \sinh{(\beta_{i}+\phi)} \\
         \cosh{(\beta_{i}+\phi)}
       \end{pmatrix}\\
       &+\rho\phi''\begin{pmatrix}
         \sinh{(\beta_{i}+\phi)} \\
         \cosh{(\beta_{i}+\phi)}
       \end{pmatrix}+\rho(\phi')^{2}\begin{pmatrix}
         \cosh{(\beta_{i}+\phi)} \\
         \sinh{(\beta_{i}+\phi)}
       \end{pmatrix},
   \end{align*}
   using that the determinant of a matrix with two identical columns is zero, we can rewrite (\ref{Wedge q2}) as
   \begin{align*}
     q_{i3}\ddot{q}_{i4}-q_{i4}\ddot{q}_{i3}&=0+-\rho\left(2\rho'\phi'+\rho\phi''\right)\det{\begin{pmatrix}
         \cosh{(\beta_{i}+\phi)} & \sinh{(\beta_{i}+\phi)} \\
         \sinh{(\beta_{i}+\phi)} & \cosh{(\beta_{i}+\phi)}
       \end{pmatrix}}\\
       &=-\rho\left(2\rho'\phi'+\rho\phi''\right)\cdot(1).
   \end{align*}
   So combined with (\ref{Wedge q}), we get
   \begin{align*}
     0e_{3}\wedge e_{4}=-\rho\sum\limits_{i=1}^{n}m_{i}(2\rho'\phi'+\rho\phi'')e_{3}\wedge e_{4}=-\rho(2\rho'\phi'+\rho\phi'')\sum\limits_{i=1}^{n}m_{i}e_{3}\wedge e_{4},
   \end{align*}
   giving that indeed $2\rho'\phi'+\rho\phi''=0$.
  \end{proof}
  \section{Proof of Theorem~\ref{Main Theorem 1}}\label{Section proof of main theorem 1}
  Let $q_{1}$,...,$q_{n}$ be a negative hyperbolic rotopulsator as in Definition~\ref{Definition negative hyperbolic}.
     Let $I$ be the $2\times 2$ identity matrix. Let $\rho_{i}=\rho$ be independent of $i$ for all $i\in\{1,...,n\}$. Then
     \begin{align}\label{q_i negative hyperbolic*}
       \begin{pmatrix}
         q_{i3} \\
         q_{i4}
       \end{pmatrix}=\rho S(\phi+\beta_{i})\begin{pmatrix}
         0 \\
         1
       \end{pmatrix}
     \end{align}
     and inserting (\ref{q_i negative hyperbolic*}) into (\ref{EquationsOfMotion Curved}) and multiplying both sides of the resulting system of equations for the third and fourth coordinates of $q_{i}$ from the left by $S(\phi+\beta_{i})^{-1}$
    gives, as $\sigma=-1$,
    \begin{align*}
      &\left(\rho''I+(2\rho'\phi'+\rho\phi'')\begin{pmatrix}
        0 & 1 \\
        1 & 0
      \end{pmatrix}+\rho(\phi')^{2}I\right)\begin{pmatrix}
        0 \\
        1
      \end{pmatrix}\nonumber\\
      &=\sum\limits_{j=1, j\neq i}^{n}\frac{m_{j}\rho\left(\begin{pmatrix}
        \sinh(\beta_{j}-\beta_{i}) \\
        \cosh(\beta_{j}-\beta_{i})
      \end{pmatrix}+(q_{i}\odot q_{j})\begin{pmatrix}
        0\\
        1
      \end{pmatrix}\right)}{((q_{i}\odot q_{j})^{2}-1)^{\frac{3}{2}}}\\
      &+((q_{i1}')^{2}+(q_{i2}')^{2}-((\rho')^{2}+\rho^{2}(\phi')^{2}))\rho\begin{pmatrix}
        0 \\
        1
      \end{pmatrix},
    \end{align*}
    which can be rewritten as
     \begin{align}\label{Second two equations positive positive + positive negative}
      &\left(\rho''I+(2\rho'\phi'+\rho\phi'')\begin{pmatrix}
        0 & 1 \\
        1 & 0
      \end{pmatrix}+\rho(\phi')^{2}I-\rho((q_{i1}')^{2}+(q_{i2}')^{2}-((\rho')^{2}+\rho^{2}(\phi')^{2}))I\right)\begin{pmatrix}
        0 \\
        1
      \end{pmatrix}\nonumber\\
      &=\sum\limits_{j=1, j\neq i}^{n}\frac{m_{j}\rho\left(\begin{pmatrix}
        \sinh(\beta_{j}-\beta_{i}) \\
        \cosh(\beta_{j}-\beta_{i})
      \end{pmatrix}+(q_{i}\odot q_{j})\begin{pmatrix}
        0\\
        1
      \end{pmatrix}\right)}{((q_{i}\odot q_{j})^{2}-1)^{\frac{3}{2}}}.
    \end{align}
    Collecting terms for the first coordinate on both sides of (\ref{Second two equations positive positive + positive negative}) gives
    \begin{align}\label{Oh yeah}
      2\rho'\phi'+\rho\phi''=\sum\limits_{j=1, j\neq i}^{n}\frac{m_{j}\rho\sinh(\beta_{j}-\beta_{i})}{((q_{i}\odot q_{j})^{2}-1)^{\frac{3}{2}}}.
    \end{align}
    Relabeling $\beta_{1}$,...,$\beta_{n}$ if necessary, let $\beta_{1}=\min\{\beta_{j}|j\in\{1,...,n\}\}$. Then by Lemma~\ref{Lemmaaaaa} we have that $2\rho'\phi'+\rho\phi''=0$ and as $(\sigma-\sigma(q_{i}\odot q_{j})^{2})^{\frac{3}{2}}>0$, that
    \begin{align}\label{Identity fourth coordinate}
      0=\sum\limits_{j=1, j\neq 1}^{n}\frac{m_{j}\rho\sinh(\beta_{j}-\beta_{1})}{((q_{1}\odot q_{j})^{2}-1)^{\frac{3}{2}}}\geq 0.
    \end{align}
    Note that $\sinh(\beta_{j}-\beta_{1})=0$ if and only if $\beta_{j}=\beta_{1}$ and that \begin{align*}(\sigma-\sigma(q_{i}\odot q_{j})^{2})^{\frac{3}{2}}>0.\end{align*} Then for (\ref{Identity fourth coordinate}) to hold for $i=1$, all $\beta_{j}$ have to be equal to $\beta_{1}$. This proves that
    \begin{align*}
      \begin{pmatrix}
        q_{i3} \\
        q_{i4}
      \end{pmatrix}=\rho S(\beta_{1}+\phi)\begin{pmatrix}
        0 \\
        1
      \end{pmatrix}.
    \end{align*}
    This completes the proof.
    \section{Proof of Theorem~\ref{Main Theorem 2}}\label{Section proof of main theorem 2}
    Let $q_{1}$,...,$q_{n}$ be a positive elliptic, or negative elliptic rotopulsator for which $r_{i}=r$ for all $i\in\{1,...,n\}$, $q_{i3}$ and $q_{i4}$ independent of $i$ for all $i\in\{1,...,n\}$. It was proven in \cite{T2} that such a rotopulsator, for $r$ not constant, has to form a regular polygon. That means that $\alpha_{j}-\alpha_{i}=\frac{2\pi}{n}(j-i)$ for all $i$, $j\in\{1,...,n\}$. It was proven in \cite{T2}, Criterion~1 that showing existence of these positive elliptic, or negative elliptic rotopulsators is equivalent with, accounting for a change of notation, showing the existence of $\alpha_{i}$ and $r$ solving
    \begin{align}
      &b_{i}=\sum\limits_{j=1,j\neq i}^{n}\frac{m_{j}(1-\cos{(\alpha_{j}-\alpha_{i})})^{-\frac{1}{2}}}{(2-\sigma r^{2}(1-\cos{(\alpha_{j}-\alpha_{i})}))^{\frac{3}{2}}},\label{Such a mass1}\\
      &0=\sum\limits_{j=1,j\neq i}^{n}\frac{m_{j}\sin{(\alpha_{j}-\alpha_{i})}}{(1-\cos{(\alpha_{j}-\alpha_{i})})^{\frac{3}{2}}(2-\sigma r^{2}(1-\cos{(\alpha_{j}-\alpha_{i})}))^{\frac{3}{2}}},\label{Such a mass2}\\
      &b_{1}=...=b_{n}\label{Such a mass3}.
    \end{align}
    So combining (\ref{Such a mass1}), (\ref{Such a mass2}) and (\ref{Such a mass3}) with the fact that $\alpha_{j}-\alpha_{i}=\frac{2\pi}{n}(j-i)$, we get that
    \begin{align*}
      &b_{i}=\sum\limits_{j=1,j\neq i}^{n}\frac{m_{j}(1-\cos{\frac{2\pi}{n}(j-i)})^{-\frac{1}{2}}}{(2-\sigma r^{2}(1-\cos{\frac{2\pi}{n}(j-i)}))^{\frac{3}{2}}}, \\
      &0=\sum\limits_{j=1,j\neq i}^{n}\frac{m_{j}\sin{\frac{2\pi}{n}(j-i)}}{(1-\cos{\frac{2\pi}{n}(j-i)})^{\frac{3}{2}}(2-\sigma r^{2}(1-\cos{\frac{2\pi}{n}(j-i)}))^{\frac{3}{2}}}, \\
      &b_{1}=...=b_{n},
    \end{align*}
    which can be rewritten, defining $m_{j+kn}=m_{j}$ for $k\in\mathbb{Z}$, $j\in\{1,...,n\}$ as
    \begin{align}
      &b_{i}=\sum\limits_{j=1}^{n-1}\frac{m_{j+i}(1-\cos{\frac{2\pi}{n}j})^{-\frac{1}{2}}}{(2-\sigma r^{2}(1-\cos{\frac{2\pi}{n}j}))^{\frac{3}{2}}},\label{mass1}\\
      &0=\sum\limits_{j=1}^{n-1}\frac{m_{j+i}\sin{\frac{2\pi}{n}j}}{(1-\cos{\frac{2\pi}{n}j})^{\frac{3}{2}}(2-\sigma r^{2}(1-\cos{\frac{2\pi}{n}j}))^{\frac{3}{2}}},\label{mass2} \\
      &b_{1}=...=b_{n}.\label{mass3}
    \end{align}
    Let $j_{1}$, $j_{2}\in\{1,...,n-1\}$. Note that two terms
    \begin{align*}
      \frac{m_{j_{1}+i}\sin{\frac{2\pi}{n}j_{1}}}{(1-\cos{\frac{2\pi}{n}j_{1}})^{\frac{3}{2}}(2-\sigma r^{2}(1-\cos{\frac{2\pi}{n}j_{1}}))^{\frac{3}{2}}}
    \end{align*}
    and
    \begin{align*}
      \frac{m_{j_{2}+i}\sin{\frac{2\pi}{n}j_{2}}}{(1-\cos{\frac{2\pi}{n}j_{2}})^{\frac{3}{2}}(2-\sigma r^{2}(1-\cos{\frac{2\pi}{n}j_{2}}))^{\frac{3}{2}}}
    \end{align*}
    are linearly independent if and only if $\cos{\frac{2\pi}{n}j_{1}}\neq\cos{\frac{2\pi}{n}j_{2}}$, because $r$ is not constant. As
    \begin{align*}
      \cos{\frac{2\pi}{n}j_{1}}=\cos{\frac{2\pi}{n}j_{2}}\textrm{ if and only if }\frac{2\pi}{n}j_{1}=\frac{2\pi}{n}j_{2}\textrm{ or }\frac{2\pi}{n}j_{1}=2\pi-\frac{2\pi}{n}j_{2}=\frac{2\pi}{n}(n-j_{2}),
    \end{align*}
    by (\ref{mass2}) this means that for each $j\in\{1,...,n-1\}$ we have that
    \begin{align*}
      0&=\frac{m_{j+i}\sin{\frac{2\pi}{n}j}}{(1-\cos{\frac{2\pi}{n}j})^{\frac{3}{2}}(2-\sigma r^{2}(1-\cos{\frac{2\pi}{n}j}))^{\frac{3}{2}}}\\
      &+\frac{m_{n-j+i}\sin{\frac{2\pi}{n}(n-j)}}{(1-\cos{\frac{2\pi}{n}(n-j)})^{\frac{3}{2}}(2-\sigma r^{2}(1-\cos{\frac{2\pi}{n}(n-j)}))^{\frac{3}{2}}}\\
      &=(m_{j+i}-m_{n-j+i})\frac{\sin{\frac{2\pi}{n}j}}{(1-\cos{\frac{2\pi}{n}j})^{\frac{3}{2}}(2-\sigma r^{2}(1-\cos{\frac{2\pi}{n}j}))^{\frac{3}{2}}}.
    \end{align*}
    So taking $j=1$, this means that $m_{1+i}=m_{-1+i}$. So for $n$ odd, this means that all masses are equal. For $n$ even, additional work is required: Because $m_{1+i}=m_{-1+i}$, we may write that $m_{j}=m$ for the even-labeled masses and $m_{j}=M$ for the odd-labeled masses, giving by (\ref{mass1}) for $i=1$ and $i=2$
    \begin{align*}
      b_{1}&=\sum\limits_{j\textrm{ even }}\frac{M(1-\cos{\frac{2\pi}{n}j})^{-\frac{1}{2}}}{(2-\sigma r^{2}(1-\cos{\frac{2\pi}{n}j}))^{\frac{3}{2}}}   +\sum\limits_{j\textrm{ odd }}\frac{m(1-\cos{\frac{2\pi}{n}j})^{-\frac{1}{2}}}{(2-\sigma r^{2}(1-\cos{\frac{2\pi}{n}j}))^{\frac{3}{2}}}\textrm{ and } \\
      b_{2}&=\sum\limits_{j\textrm{ even }}\frac{m(1-\cos{\frac{2\pi}{n}j})^{-\frac{1}{2}}}{(2-\sigma r^{2}(1-\cos{\frac{2\pi}{n}j}))^{\frac{3}{2}}}
      +\sum\limits_{j\textrm{ odd }}\frac{M(1-\cos{\frac{2\pi}{n}j})^{-\frac{1}{2}}}{(2-\sigma r^{2}(1-\cos{\frac{2\pi}{n}j}))^{\frac{3}{2}}}
    \end{align*}
    respectively, thus giving by (\ref{mass3}) that
    \begin{align}
      0&=b_{1}-b_{2}=(M-m)\left(\sum\limits_{j\textrm{ even }}\frac{(1-\cos{\frac{2\pi}{n}j})^{-\frac{1}{2}}}{(2-\sigma r^{2}(1-\cos{\frac{2\pi}{n}j}))^{\frac{3}{2}}}\right)\nonumber \\
      &-(M-m)\left(\sum\limits_{j\textrm{ odd }}\frac{(1-\cos{\frac{2\pi}{n}j})^{-\frac{1}{2}}}{(2-\sigma r^{2}(1-\cos{\frac{2\pi}{n}j}))^{\frac{3}{2}}}\right).\label{Presto}
    \end{align}
    Note that, again by linear independence, the terms in
    \begin{align*}
      \sum\limits_{j\textrm{ even }}\frac{(1-\cos{\frac{2\pi}{n}j})^{-\frac{1}{2}}}{(2-\sigma r^{2}(1-\cos{\frac{2\pi}{n}j}))^{\frac{3}{2}}}
    \end{align*}
    cannot cancel out against the terms in
    \begin{align*}
      \sum\limits_{j\textrm{ odd }}\frac{(1-\cos{\frac{2\pi}{n}j})^{-\frac{1}{2}}}{(2-\sigma r^{2}(1-\cos{\frac{2\pi}{n}j}))^{\frac{3}{2}}},
    \end{align*}
    meaning that by (\ref{Presto}) we have that $M=m$. This completes the proof.
    \section{Proof of Corollary~\ref{Main Theorem 3}}\label{Section proof of main theorem 3}
    By Theorem~\ref{Main Theorem 2}, all positive elliptic and negative elliptic rotopulsators that are not relative equilibria, for which the $q_{i3}$ and $q_{i4}$, $i\in\{1,...,n\}$, are independent of $i$,  have to have a regular polygon configuration and equal masses. By Theorem~\ref{Main Theorem 1}, negative hyperbolic rotopulsators for which $\rho_{i}$, $i\in\{1,...,n\}$, is independent of $i$ have to be negative elliptic to have a point configuration that maintains its shape. What remains to show is that for $r$ constant there exists at most one relative equilibrium for each set of masses if $\sigma=-1$, that  the same is true for $\sigma=1$ if $r<\frac{2}{5}\sqrt{5}$ and finally that these rotopulsators exist. We will begin with the latter:
    If a positive elliptic rotopulsator, or a negative elliptic rotopulsator $q_{1},...,q_{n}$ for which $r_{i}=r$ is independent of $i$ for $i\in\{1,...,n\}$, as are $q_{i3}=z_{1}$ and $q_{i4}=z_{2}$, then to prove existence of such a rotopulsator, we need to prove that it solves (\ref{Such a mass1}), (\ref{Such a mass2}) and (\ref{Such a mass3}). Because we already know that these rotopulsators have a configuration of a regular polygon and have equal masses, it suffices to show that (\ref{mass1}), (\ref{mass2}) and (\ref{mass3}) are fulfilled. If the $m_{j+i}$ in (\ref{mass1}) are all equal, then the $b_{i}$ in (\ref{mass1}) and therefore in (\ref{mass3}) are all equal, which leaves establishing (\ref{mass2}): If all masses are equal, then writing $k=n-j$ and $m_{j+i}=m$, we get that
    \begin{align*}
      &\sum\limits_{j=1}^{n-1}\frac{m\sin{\frac{2\pi}{n}j}}{(1-\cos{\frac{2\pi}{n}j})^{\frac{3}{2}}(2-\sigma r^{2}(1-\cos{\frac{2\pi}{n}j}))^{\frac{3}{2}}}\\
      &=\sum\limits_{k=1}^{n-1}\frac{m\sin{\frac{2\pi}{n}(n-k)}}{(1-\cos{\frac{2\pi}{n}(n-k)})^{\frac{3}{2}}(2-\sigma r^{2}(1-\cos{\frac{2\pi}{n}(n-k)}))^{\frac{3}{2}}} \\
      &=-\sum\limits_{k=1}^{n-1}\frac{m\sin{\frac{2\pi}{n}k}}{(1-\cos{\frac{2\pi}{n}k})^{\frac{3}{2}}(2-\sigma r^{2}(1-\cos{\frac{2\pi}{n}k}))^{\frac{3}{2}}},
    \end{align*}
    giving that indeed
    \begin{align*}
      \sum\limits_{j=1}^{n-1}\frac{m_{j+i}\sin{\frac{2\pi}{n}j}}{(1-\cos{\frac{2\pi}{n}j})^{\frac{3}{2}}(2-\sigma r^{2}(1-\cos{\frac{2\pi}{n}j}))^{\frac{3}{2}}}=0.
    \end{align*}
    This proves that these rotopulsators indeed exist.
    To prove that for fixed masses there exists at most one negative elliptic relative equilibrium and for $\sigma=1$ if additionally $r<\frac{2}{5}\sqrt{5}$ there exists at most one positive elliptic relative equilibrium, we will use the following result from \cite{T5}: Let $0\leq\alpha_{1}<...<\alpha_{n}<2\pi$, $m_{1}>0$,...,$m_{n}>0$, $r>0$ and $A>0$ be constants. Then there exists at most one relative equilibrium $q_{1},...,q_{n}\in\mathbb{R}^{2}$, \begin{align*}q_{i}(t)=rT(\alpha_{i}+At)\begin{pmatrix}
      1\\ 0
    \end{pmatrix}\end{align*}  of
    \begin{align}\label{Alternative equations of motion}
      \ddot{q}_{i}=\sum\limits_{j=1,j\neq i}m_{j}(q_{j}-q_{i})f(\|q_{j}-q_{i}\|),\textrm{ }i\in\{1,...,n\},
    \end{align}
    where
    $f$ is a positive, differentiable function for which $\frac{d}{dx}(xf(x))<0$ for each fixed set of masses $m_{1}$,...,$m_{n}$. \newline Inserting $q_{i}(t)=rT(\alpha_{i}+At)\begin{pmatrix}
      1\\ 0
    \end{pmatrix}$ into (\ref{Alternative equations of motion}) and multiplying the resulting equations from the left by $T(\alpha_{i}+At)^{-1}$ shows that this is equivalent with showing that, for $m_{1}$,...,$m_{n}$ fixed, the system
    \begin{align*}
      A^{2}\begin{pmatrix}
        1 \\ 0
      \end{pmatrix}=\sum\limits_{j=1,j\neq i}^{n}m_{j}\begin{pmatrix}
        1-\cos{(\alpha_{j}-\alpha_{i})} \\
        \sin{(\alpha_{j}-\alpha_{i})}
      \end{pmatrix}f(\sqrt{2}r\sqrt{1-\cos{(\alpha_{j}-\alpha_{i})}}),
    \end{align*}
    has at most one solution $\alpha_{1}$,...,$\alpha_{n}$, $r$, barring of course the possibility of constant rotations of relative equilibria.
    If $q_{1}$,...,$q_{n}$ is a positive elliptic, or negative elliptic relative equilibrium of (\ref{EquationsOfMotion Curved}) for which $q_{i3}=z_{1}$ and $q_{i4}=z_{2}$ is independent of $i$ for all $i\in\{1,...,n\}$, then $r$ is constant and we may assume that $\theta(t)=At$ for a certain positive constant $A>0$, meaning that by (2.24) in the proof of Criterion~1 of \cite{T2} we get, adjusting for notation and if $1-\sigma r^{2}\neq 0$, that
    \begin{align*}
      r^{3}\begin{pmatrix}
        A^{2} \\
        0
      \end{pmatrix}=\sum\limits_{j=1,j\neq i}^{n}m_{j}\begin{pmatrix}
        1-\cos{(\alpha_{j}-\alpha_{i})} \\
        \sin{(\alpha_{j}-\alpha_{i})}
      \end{pmatrix}f(\sqrt{2}r\sqrt{1-\cos{(\alpha_{j}-\alpha_{i})}}).
    \end{align*}
    with $f(x)=(x^{2}-\frac{1}{4}\sigma x^{4})^{-\frac{3}{2}}$. For $\sigma=-1$, this means that $\frac{d}{dx}(xf(x))<0$. For $\sigma=1$ we get that $\frac{d}{dx}(xf(x))<0$ if and only if $x^{2}<\frac{5}{8}$. As $xf(x)$ need only be decreasing for any values $x=\sqrt{2}r\sqrt{1-\cos{(\alpha_{j}-\alpha_{i})}}$, we may assume that $x<2r$, meaning that $4r^{2}<\frac{5}{8}$, or $r<\frac{2}{5}\sqrt{5}$. This completes the proof.

\end{document}